\documentclass[a4paper]{amsart}

\usepackage{a4}

\begin{document}
\newcommand{\J}{\mathcal{J}} 
\newcommand{\cL}{\mathcal{L}}
\newcommand{\I}{\mathcal{I}}
\newcommand{\R}{\mathcal{R}}
\newcommand{\A}{\mathcal{A}}
\newcommand{\F}{\mathcal{F}}
\newcommand{\B}{\mathcal{B}}
\newcommand{\T}{\mathcal{T}}
\newcommand{\C}{\mathcal{C}}
\newcommand{\D}{\mathcal{D}}
\newcommand{\Q}{\mathcal{Q}}
\newcommand{\cS}{\mathcal{S}}
\newcommand{\FF}{\mathbb{F}}
\newcommand{\ov}{\overline}
\newcommand{\wh}{\widehat}
\newtheorem{theorem}{Theorem}[section]
\newtheorem{proposition}[theorem]{Proposition}
\newtheorem{lemma}[theorem]{Lemma}
\newtheorem{corollary}[theorem]{Corollary}

\theoremstyle{remark}
\newtheorem{remark}[theorem]{Remark}
\newtheorem{definition}[theorem]{Definition}
\newtheorem{example}[theorem]{Example}
\newtheorem{examples}[theorem]{Examples}
\newtheorem{question}[theorem]{Question}

\title[An elementary approach to Wedderburn's structure theory]
{An elementary approach to Wedderburn's structure theory}
\thanks{Supported by the Slovenian Research Agency (program No. P1-0288).
}
\author{Matej Bre\v sar}
\address{Faculty of Mathematics and Physics, University of Ljubljana, and Faculty of Natural Sciences and Mathematics,
 University of Maribor, Slovenia	}
\email{matej.bresar@fmf.uni-lj.si}
\date{}

\begin{abstract}
Wedderburn's theorem on the structure of  finite dimensional semisimple algebras  is proved by using minimal prerequisites.
\end{abstract}

\maketitle

\section{Introduction}

One hundred years have passed since  J.\,H.\,M.  Wedderburn published the paper on the structure of  finite dimensional semisimple algebras  \cite{W}. He proved that every such algebra is a direct product of matrix algebras over division algebras. Versions of this result appear in countless graduate algebra textbooks, often on their first pages. Standard proofs are based on the concept of a module over an algebra. The module-theoretic approach is certainly elegant and efficient, and moreover it gives a basis for developing various more general theories. However, Wedderburn's theorem  has a very simple formulation which  does not involve modules. Therefore it seems natural to seek for more direct approaches.  The goal of this article is to present a proof which uses only the most elementary tools.  It is short, but so are some module-theoretic proofs. Its main advantage is the conceptual simplicity. It cannot replace standard proofs if one has a development of a more sophisticated theory in mind.    But it might be more easily accessible to students.    Wedderburn's theorem is a typical graduate level topic, but using this approach it could be included in an undergraduate algebra  course. Students often find an introduction to algebra somewhat dry and formal, and therefore enliven it with colorful  theorems might make them more interested. We wish to show that  Wedderburn's  beautiful and important theorem is one such option.
   
In Section \ref{sprer} we give a self-contained  introduction to the theory of  (associative) algebras. 
While it is certainly somewhat fragmentary and condensed, it does not include only bare definitions, but also examples and informal comments.
 Along the way we establish  two simple and well-known lemmas that are needed later. 
 The author is aware that many of the potential readers of this article know the contents of Section 2 very well, and therefore they might find the inclusion of basic definitions and examples redundant. But what we wish to point out is  that in principle Wedderburn's theorem and its proof can be understood even by  a student 
familiar with basic concepts of  linear algebra. Therefore we have decided, in spite of the risk of appearing  na\"  ive, 
to assume the knowledge on vector spaces and matrices as the only 
background needed to follow this paper.

  In Section 3 we prove Wedderburn's theorem first for prime algebras with unity, then for general prime algebras, and finally for semiprime (i.e., semisimple) algebras. The reason for dealing with prime algebras instead of with (more common but less general) simple ones is not because of seeking for a greater level of generality, but because  the proofs run more smoothly in this setting.

\section{Prerequisites}\label{sprer}

A vector space $\A$ over a field $\FF$ is called an {\em algebra over $\FF$} if it is equipped with a map $\A\times\A\to \A$, 
denoted $(a,b)\mapsto ab$ and called  {\em multiplication}, that is bilinear and associative. This means that 
$$a(b+c)= ab + ac,\, (b+c)a=ba  + ca,\, \lambda(ab)=  (\lambda a)b = a(\lambda b),\,(ab)c = a(bc)$$
 for all $a,b,c\in \A$ and $\lambda\in\FF$. The element $ab$ is called the {\em product} of $a$ and $b$. 
  Setting $b=c=0$ in the first two identities we get $a0=0a=0$ for every $a\in\A$. Because  of the associativity it makes sense to introduce the notation $abc$ for the product of three elements; this can be interpreted either as $(ab)c$ or $a(bc)$. The meaning of  $a_1a_2\ldots a_n$, where $a_i\in\A$ and $n\ge 1$, should then be self-explanatory.
  If $ab=ba$ for all $a,b\in\A$, then $\A$ is said to be a {\em commutative algebra}.
 
 Unless stated otherwise, by an algebra we shall always mean an algebra over a (fixed) field $\FF$. The simplest example of an algebra is $\{0\}$; by a {\em nonzero algebra} we mean any algebra that contains a nonzero element.
 An algebra is said to be {\em finite dimensional} if it is finite dimensional as a vector space.  If $\B$ is a linear subspace of $\A$ such that $bb'\in\B$ whenever $b,b'\in\B$, then we call $\B$  a {\em subalgebra} of $\A$. It is clear that $\B$ itself is an algebra. If $a,b$ are elements of an algebra $\A$, then $a\A = \{ax\,|\,x\in\A\}$,  $\A b =\{xb\,|\,x\in\A\}$ and
 $a\A b =\{axb\,|\,x\in\A\}$
 are subalgebras of $\A$.

For every element $a$ in an algebra $\A$ we write 
$a^2$ for $aa$. We call $e\in\A$ an {\em idempotent} if $e^2=e$. For example, $0$ is an idempotent. Another basic example is a {\em unity}  element. This is  
a nonzero element $1_\A\in\A$ such that $1_\A a = a1_\A = a$ for all $a\in\A$.  We do not, however, require that our algebras must contain a unity element. Every idempotent different from $0$ and 
$1_\A$ is called a {\em nontrivial idempotent}. If $e$ is a nontrivial idempotent, then so is $1_\A-e$. We remark that $e(1_\A-e)=(1_\A-e)e=0$.

 If   $\A$ is an  algebra with unity, then $a\in\A$ is said to be {\em invertible} if there exists 
$b\in\A$ such that $ab = ba =1_\A$. If $a,b,c\in\A$ are such that
$ab = bc =1_\A$, then  $c= 1_\A c = (ab)c= a (bc)= a 1_\A =a $, and so $a$ is invertible.
Note that $b$ satisfying $ab=ba=1$ is unique. We denote it by $a^{-1}$.  
 A {\em division algebra} is an algebra with unity  in which every nonzero element is invertible. For example, $\FF$ is a $1$-dimensional division algebra over $\mathbb F$. We remark, however, that a commutative division algebra over $\FF$ is of course  a field itself, but not necessarily of dimension $1$ over $\FF$. For instance, $\mathbb C$ is a   $1$-dimensional division algebra over $\mathbb C$, a $2$-dimensional division algebra over $\mathbb R$, and an infinite dimensional algebra over $\mathbb Q$. There  exist  division algebras that are not commutative. The simplest (and the oldest - discovered by W. R. Hamilton in 1843!) example is the algebra $\mathbb H$ of real quaternions. As a vector space  $\mathbb H$ is a $4$-dimensional space over $\mathbb R$ with basis traditionally denoted by  $1,i,j,k$. Obviously, the multiplication in $\mathbb H$ is completely determined by products of basis elements. These are defined as follows: $i^2 = j^2 = k^2=-1$, $ij=-ji= k$, $jk = -kj = i$,
$ki=-ik = j$, and, as the notation suggests, $1$ is the unity element of $\mathbb H$. One can easily check that $\mathbb H$ is indeed an algebra, and moreover, if $h = a + bi + cj + dk$ where at least one of $a,b,c,d\in\mathbb R$ is nonzero, then $h$ is invertible with $h^{-1} = (a^2 + b^2 + c^2 + d^2)^{-1}(a - bi - cj - dk)$. 
Thus $\mathbb H$ is a noncommutative division algebra.

Note that  a division algebra $\D$ cannot contain nontrivial idempotents. Moreover, the product of nonzero elements in $\D$ is always nonzero.  
Division algebras  are  too ``perfect" to indicate the challenges of the general theory of algebras. The next example is more illustrative.
Let $M_n(\FF)$ denote the set of all $n\times n$ matrices $(a_{ij})$ with entries $a_{ij}$ belonging to  $\FF$. Then $M_n(\FF)$ becomes a (finite dimensional) algebra if we define addition, scalar multiplication, and multiplication in the usual way, i.e.,
$(a_{ij}) + (b_{ij}) = (a_{ij}+ b_{ij})$, $\lambda(a_{ij}) = (\lambda a_{ij})$, and $(a_{ij})(b_{ij}) = (c_{ij})$ where
$c_{ij}=\sum_{k=1}^n a_{ik}b_{kj}$. 
These operations  make sense in a more general context where the entries $a_{ij},b_{ij}$ are not scalars but elements of an arbitrary algebra $\C$. The algebra axioms  remain fulfilled in this setting. Thus, the set $M_n(\C)$ of all $n\times n$ matrices $(a_{ij})$ with $a_{ij}\in \C$ is an algebra under the standard matrix operations.

Let $\A$ be an algebra with unity and let $n\ge 1$. Elements $e_{ij}\in\A$, $i,j=1,\ldots,n,$ are called {\em matrix units} if $e_{11}+\ldots+e_{nn} = 1_\A$ and $e_{ij}e_{kl}=\delta_{jk}e_{il}$ for all $i,j,k,l$ (here, $\delta_{jk}$ is the ``Kronecker delta"). In particular, $e_{ii}$ are idempotents such that $e_{ii}e_{jj} =0$ if $i\ne j$. One can check that each $e_{ij}\ne 0$. If $\A=M_n(\C)$ where $\C$ is an algebra with unity, then $\A$ has matrix units. Indeed, the standard (but not the only) example is the following: $e_{ij}$ is the matrix whose $(i,j)$-entry is $1_\C$ and all other entries are $0$. In our first  lemma below we will show that $M_n(\C)$ is basically also the only example of an algebra with matrix units. But first we have to introduce another  concept which serves as a tool for discovering which  familiar  algebra is  hidden   behind the algebra in question.

Let $\A$ and $\B$ be algebras. We say that $\A$ and $\B$ are {\em isomorphic}, and we write $\A\cong\B$, if there exists a bijective linear map $\varphi:\A\to \B$ such that $\varphi(ab)=\varphi(a)\varphi(b)$ for all $a,b\in\A$. Such a map $\varphi$ is called an {\em isomorphism}.  Isomorphic algebras have exactly the same properties; informally we consider them as identical, although they may  appear very different
at a glance. For example, consider the subalgebra $\A$ of $M_2(\mathbb R)$ consisting of all matrices of the form 
$\begin{bmatrix}
a & b  \\
-b & a
\end{bmatrix}$ with $a,b\in\mathbb R$.   It is easy to see that  the map $\begin{bmatrix}
a & b  \\
-b & a
\end{bmatrix}\mapsto  a + bi$ is an isomorphism from $\A$ onto $\mathbb C$ (here  $\mathbb C$ is considered as an algebra over $\mathbb R$). Thus, as long as we are interested only in addition and (scalar) multiplication,  $\A$ is just a disguised form of more familiar complex numbers.

\begin{lemma}\label{LMU}
Let $\A$ be an algebra with unity. If $\A$ contains matrix units  $e_{ij}$, $i,j=1,\ldots n$, then $\A\cong M_n(e_{tt}\A e_{tt})$ for each $t=1,\ldots,n$. 
\end{lemma}

\begin{proof}
For every $a\in\A$ we set $a_{ij} = e_{ti}ae_{jt}$. We can also write $a_{ij} = e_{tt}e_{ti}ae_{jt}e_{tt}$ and so $a_{ij}\in e_{tt}\A e_{tt}$. 
 Now define $\varphi:\A\to M_n(e_{tt}\A e_{tt})$ by $\varphi(a)=(a_{ij})$. The linearity
of $\varphi$ is clear. The $(i,j)$-entry of $\varphi(a)\varphi(b)$ is equal to $\sum_{k=1}^n e_{ti}ae_{kt}e_{tk}be_{jt} =e_{ti}a(\sum_{k=1}^n e_{kk})be_{jt} = e_{ti}abe_{jt},$ which is the $(i,j)$-entry of $\varphi(ab)$. Thus, $\varphi(ab)=\varphi(a)\varphi(b)$. If $a_{ij}=0$ for all $i,j$, then $e_{ii}ae_{jj} = e_{it}a_{ij}e_{tj} =0$, and so $a=0$ since the sum of all $e_{ii}$ is $1_\A$. Thus $\varphi$ is injective.
Checking  the surjectivity is also easy.
\end{proof}  

An algebra $\A$ is said to be {\em prime} if for all $a,b\in\A$, $a\A b = \{0\}$ implies $a=0$ or $b=0$. 
If $\D$ is a division algebra, then  $M_n(\D)$ is a prime algebra for every $n\ge 1$. Indeed,  if  $a,b\in  M_n(\D)$ are such that 
 $ae_{ij} b = 0$ for all standard matrix units $e_{ij}$, then  $a=0$ or $b=0$. The proof is an easy exercise and we omit details. Next, an algebra $\A$ is said to be {\em semiprime} if for all $a\in\A$, $a\A a = \{0\}$ implies $a=0$.
Obviosuly, prime algebras are semiprime. The converse is not true, as we shall see in the next paragraph.  

Let $\A_1$ and $\A_2$ be algebras. Then their Cartesian product $\A_1\times \A_2$ becomes an algebra by definining operations in the following natural way: 
\begin{align*}
(a_1,a_2) + (b_1,b_2) &= (a_1+b_1,a_2+b_2),\\
\lambda(a_1,a_2)&=(\lambda a_1,\lambda a_2),\\
(a_1,a_2)(b_1,b_2) &= (a_1b_1,a_2b_2).
\end{align*}
This algebra, which we denote just as the Cartesian product by $\A_1\times\A_2$, is called the {\em direct product} of algebras $\A_1$ and $\A_2$. Similarly we define the direct product $\A_1\times\ldots\times\A_n$ of any finite family of algebras. 
Just as the product of two natural numbers  different from $1$ is not a prime number, the direct product of two algebras different from $\{0\}$ is not a prime algebra. However, the direct product of semiprime algebras is a semiprime algebra. An example of an algebra that is not semiprime is the algebra $T_n(\FF)$ of all upper triangular matrices in $M_n(\FF)$. Note that this is indeed an algebra, i.e., a subalgebra of $M_n(\FF)$, and that $e_{ij}T_n(\FF)e_{ij}=\{0\}$ if $i < j$. 

A linear subspace $\cL$ of an algebra $\A$ is called a {\em left ideal} of $\A$ if $a\cL\subseteq \cL$ for every $a\in\A$. For example, $\A a$ is a left ideal of $\A$ for every $a\in\A$. It is easy to see that $\{0\}$ and $\D$ are the only left ideals of a division algebra $\D$. Nevertheless,  as the proof of the next lemma will show, there is an important connection between left ideals and division algebras. Before stating this lemma, we mention an illustrative example.  If $e_{tt}$ is a standard matrix unit of $M_n(\D)$, then  $e_{tt}M_n(\D)e_{tt}$ consists of all matrices whose whose $(t,t)$-entry is an arbitrary element in $\D$ and all other entries are $0$. Therefore $e_{tt}M_n(\D)e_{tt}\cong \D$.

\begin{lemma} \label{LD}
If $\A$ is a nonzero finite dimensional semiprime algebra, then there exists an idempotent $e\in\A$ such that
$e\A e$ is is a division algebra.
\end{lemma}

\begin{proof}
Pick a nonzero left ideal  $\cL$  of minimal dimension, i.e., $\dim_\FF \cL\le \dim_\FF \J$ for every nonzero left ideal $\J$ of $\A$. Obviously, $\{0\}$ is then the only left ideal that is properly contained in $\cL$. Let $0\ne x\in\cL$. Since $\A$ is semiprime, there exists $a\in\A$ such that $xax\ne 0$. As $y=ax\in\cL$, we have found $x,y\in\cL$ with $xy\ne 0$. In particular, $\cL y\ne \{0\}$. But  $\cL y$ is  a left ideal of $\A$ contained in $\cL$, and so $\cL y =\cL$. Accordingly, as $y\in\cL$ we  have $ey =y$ for some $e\in \cL$. This implies that
 $e^2-e$ belongs to the set $\J = \{z\in\cL \,|\, zy=0\}$. Clearly, $\J$ is again a left ideal of $\A$ contained in $\cL$. Since $x\in\cL\setminus{ \J}$, this time we conclude that $\J =\{0\}$. In particular, $e^2=e$. As 
$e\in\cL$, we have $\A e\subseteq \cL$, and since $0\ne e\in\A e$ it follows that
 $\cL=\A e$. Now consider the subalgebra $e\A e$ of $\A$. Obviously, $e=1_{e\A e}$. Let $a\in\A$ be such that $eae\ne 0$. The lemma will be proved by showing that $eae$ is invertible in $e\A e$. We have $\{0\}\ne \A eae \subseteq \A e =\cL$, and so $\A eae =\cL$. Therefore there is $b\in\A$ such that
 $beae = e$, and hence also $(ebe)(eae)=e$. Now, $ebe$ is again a nonzero element in $e\A e$, and so by the same argument there is $c\in\A$ such that $(ece)(ebe)=e$. But then $(eae)^{-1}= ebe$ (see the paragraph introducing division algebras).
\end{proof}

We remark that the finite dimensionality of $\A$ was used only for finding a left ideal which does not properly contain
any other nonzero left ideal. Such left ideals are called {\em minimal left ideals}. Thus, Lemma \ref{LD} also holds for a 
semiprime algebra $\A$ of arbitrary dimension which has a minimal left ideal $\cL$. The proof actually shows that $\cL$ is of the form $\cL = \A e$ where $e$ is an idempotent such that $e\A e$ is a division algebra. But we shall not need this more general result.

As one can guess, a {\em right ideal} of $\A$ is defined as  
a linear subspace $\I$ of $\A$ such that 
$\I a\subseteq \I$ for every $a\in\A$. If $\I$ is simultaneously a left ideal and a right ideal of $\A$, then $\I$ is called an {\em ideal} of $\A$. Let us give a few examples.   If $\A = \A_1\times \A_2$, then $\I = \A_1\times\{0\}$ and  $\J = \{0\}\times\A_2$ are ideals of $\A$. 
The set $S_n(\FF)$ of all strictly upper triangular matrices (i.e., matrices in $T_n(\FF)$ whose diagonal entries are $0$)  is an ideal of $T_n(\FF)$. All polynomials with coefficients in $\FF$ form an  algebra, denoted  $\FF[X]$, under standard  operations. An example of its ideal is the set of all polynomials with constant term zero.

The notions of prime and semiprime algebras can be equivalently introduced through ideals. An algebra $\A$ is prime if and only if for each pair of nonzero ideals $\I$ and $\J$ of $\A$ there exist $x\in\I$ and $y\in\J$ such that $xy\ne 0$. An algebra $\A$ is semiprime if and only if it does not contain nonzero nilpotent ideals; by a nilpotent  ideal  we mean an ideal $\I$ such that for some $n\ge 1$ we have $x_1\ldots x_n=0$ for all $x_1,\ldots,x_n\in\I$. For instance,
$S_n(\FF)$ is a nilpotent ideal of $T_n(\FF)$. 
Since we shall not need these alternative definitions, we leave the proofs as exercises for the reader.

\section{Wedderburn's theorem} \label{secW}

Let us recall our convention that an ``algebra" means an algebra over a fixed, but arbitrary field $\FF$. 

Our first goal is to prove the basic theorem of Wedderburn's structure theory, characterizing matrix algebras over division algebras through their abstract  properties. Let us first outline the concept of the proof to help the reader not to get lost in the (inevitably) tedious notation in the formal proof.
By Lemma \ref{LMU} it is enough to show that the algebra $\A$ in question contains matrix units $e_{ij}$ such that
$e_{tt}\A e_{tt}$ is a division algebra for some $t$.  
  Lemma  \ref{LD} yields the existence of an idempotent $e$ such that $e\A e$ is a divison algebra. Think of $e$ as $e_{nn}$. A simple argument based on the induction on  $\dim_\FF \A$ shows that the algebra
  $(1_\A -e)\A (1_\A -e)$ contains matrix units $e_{11},e_{12},\ldots,e_{n-1,n-1}$ with $e_{tt}\A e_{tt}$ being division algebras. It remains to find 
  $e_{n1},\ldots,e_{n,n-1}$ and $e_{1n},\ldots,e_{n-1,n}$. Finding $e_{1n}$ and $e_{n1}$ is the heart of the proof; here
  we make use of the fact that $e_{11}\A e_{11}$ and $e_{nn}\A e_{nn}$ are division algebras. The remaining matrix units 
  can be then just directly defined as 
  $e_{nj}= e_{n1}e_{1j}$ and $e_{jn}= e_{j1}e_{1n}$, $j=2,\ldots,n-1$; checking that they satisfy all 
  desired identities is straightforward.

\begin{theorem}\label{W1} {\rm (Wedderburn's theorem for prime algebras with unity)} Let $\A$ be a finite dimensional algebra with unity. Then $\A$ is prime if and only if there exist a positive integer $n$ and a  division algebra $\D$ such that
$\A\cong M_n(\D)$.
\end{theorem}

\begin{proof}
We have already mentioned that the algebra $M_n(\D)$ is prime. Therefore we only have to prove the
``only if" part. The proof is by induction on $N=\dim_\FF \A$. 

If $N=1$, then $\A = \FF 1_\A$ is a field and the result trivially holds (with $n=1$ and $\D = \FF$). 
We may therefore assume that $N > 1$. 
By Lemma \ref{LD} there exists an idempotent $e\in\A$ such that $e\A e$ is a division algebra. 
If $e=1_\A$, then the desired result  holds (with $n=1$).
 Assume therefore that $e$ is a nontrivial idempotent, and set $\wh{\A} = (1_\A -e)\A (1_\A -e)$. Note that $\wh{\A}$ is a prime algebra with unity $1_\A - e$. Further, we have  $e\wh{\A} = \wh{\A}e = \{0\}$, and so $e\notin \wh{\A}$. Therefore  $\dim_\FF \wh{\A} < N$. Using the  induction assumption it follows that  $\wh{\A}$ contains matrix units $e_{ij}$, $i,j=1,\ldots,m$, for some $m\ge 1$, such that
 $e_{ii}\wh{\A}e_{ii}$ is a division algebra for each $i$.
 Since $e_{ii}=(1_\A -e)e_{ii}=e_{ii}(1_\A -e)$, we actually have $e_{ii}\wh{\A}e_{ii}=e_{ii}\A e_{ii}$.   Our goal is  to extend these matrix units of $\wh{\A}$ to matrix units of $\A$. We begin by setting $n=m+1$ and    $e_{nn} = e$. Then
 $ e_{11}+\ldots +e_{n-1,n-1}+e_{nn} = (1_\A-e) +e =1_\A$. Using the definition of primeness twice we see that
 $e_{11}ae_{nn}a'e_{11} \ne 0$ for some $a,a'\in\A$. As $e_{11}\A e_{11}$ is a division algebra with unity $e_{11}$, it follows that  $(e_{11}ae_{nn}a'e_{11})(e_{11}a''e_{11})  = e_{11}$ for some  $a''\in\A$. Thus,
 $e_{11}ae_{nn}be_{11}=e_{11}$ where $b=a'e_{11}a''$. Let us set $e_{1n} = e_{11}ae_{nn}$ and $e_{n1}=e_{nn}be_{11}$, so that  
 $e_{1n}e_{n1}=e_{11}$. Since $e_{n1}\in e_{nn}\A e_{11}$, we have $e_{n1} = e_{nn}e_{n1}$ and 
 $e_{n1} = e_{n1}e_{11} = e_{n1}e_{1n}e_{n1}$. Comparing both relations we get $(e_{nn} -  e_{n1}e_{1n})e_{n1} =0$.
 The element  $e_{nn} -  e_{n1}e_{1n}$ lies in the division algebra  $e_{nn}\A e_{nn}$. If it is nonzero, then we can
 multiply the last identity from the left-hand side by its inverse, which gives $e_{nn}e_{n1} =0$, and hence $e_{n1}=0$ - a contradiction. Therefore $e_{nn} =  e_{n1}e_{1n}$. Finally we set
 $e_{nj}= e_{n1}e_{1j}$ and $e_{jn}= e_{j1}e_{1n}$ for $j=2,\ldots,n-1$. Note that $e_{ij} = e_{i1}e_{1j}$ then holds for all $i,j=1,\ldots,n$. Consequently, for all 
 $i,j,k,l=1,\ldots,n$ we have $e_{ij}e_{kl} = e_{i1}e_{1j}e_{k1}e_{1l} = \delta_{jk}e_{i1}e_{11}e_{1l} =\delta_{jk}e_{i1}e_{1l} = \delta_{jk}e_{il}$. Thus $e_{ij}$, $i,j=1,\ldots,n$, are indeed matrix units of $\A$.  Therefore  Lemma \ref{LMU} tells us that $\A \cong M_n(\D)$ where $\D=e_{tt}\A e_{tt}$ (for any $t$). As we know, $\D$ is a division algebra.
  \end{proof}

An obvious but important fact can be added to the statement of the theorem:  $\D$ is  also finite dimensional. 
  
 The existence of unity in Theorem \ref{W1} is actually superfluous. The necessary changes needed to cover the case without unity are not difficult, but slightly tedious. In order not to distract the reader's attention with these technicalities 
in the basic proof,  we have decided to deal with the unital case first. Besides,  to make the concept more simple and clear, many mathematicians add the existence of unity to axioms of an algebra. Each of the two settings, the unital and the non-unital, had its  advantages and its disadvantages. Our preference is the non-unital set-up, and later when dealing with semiprime algebras we shall try to demonstrate its usefulness.

Without assuming the presence of $1_\A$,  apparently we cannot introduce the algebra  $\wh{\A} = (1_\A -e)\A (1_\A -e)$ from the last proof. But actually this problem can be avoided by appropriately simulating $1_\A$. Note that 
$(1_\A -e)a (1_\A -e) = a - ea - ae + eae$, and on the right hand side $1_\A$ does not appear.
 
\begin{lemma} \label{L1}
Let $e$ be a nontrivial idempotent in a prime algebra $\A$.  Then  $\wh{\A} = \{ a - ea - ae + eae \,|\,a\in\A\}$ is a nonzero prime algebra with $e\notin \wh{\A}$. If $\wh{\A}$ has a unity, then 
$\A$ has a unity  as well. Moreover, $1_\A = e+1_{\wh{\A}}$.
\end{lemma}

\begin{proof}
Let us denote $\wh{a} = a - ea - ae + eae$.  
One can verify that for all $a,b\in\A$ we have $\wh{a}\wh{b} = \wh{ab} -\wh{aeb}$ and
$e\wh{a} = \wh{a}e =0$.
In particular, $\wh{\A}$ is a subalgebra of $\A$ with $e\notin\wh{\A}$. Further, we have $(a-ae)b(c - ec) = a\wh{b}c$.
Therefore, if $\wh{\A} =\{0\}$, then $(a-ae)\A (c - ec) =\{0\}$, and hence either $a = ae$ for every $a\in\A$ or
$c= ec\in\A$ for every $c\in\A$. But each of these two conditions implies the other one. Indeed, if, say, $a=ae$ holds, then  $a(c-ec) = (a - ae)c=0$ for all $a,c\in\A$, implying $c=ec$. Accordingly, $\wh{\A} \ne \{0\}$ since $e\ne 1_\A$ by assumption.    Next, since $b - \wh{b}\in e\A + \A e$ we have $\wh{a}(b-\wh{b})\wh{c} = 0$, i.e., $\wh{a}\wh{b}\wh{c} = \wh{a}b\wh{c}$. Therefore, the primeness of $\A$ implies the primeness of  $\wh{\A}$. 
Finally, assume that $\wh{\A}$ has a unity. Let $f\in\A$ be such that  $\wh{f}=1_{\wh{\A}}$.  Then $\wh{f}a\wh{f} = \wh{f}\wh{a}\wh{f} =\wh{a}\wh{f} = (a-ea)\wh{f}$ since $e\wh{f}=0$. Thus, we have $((e+\wh{f})a -a)\wh{f}=0$ for all $a\in\A$, and therefore
  $((e+\wh{f})a -a)b\wh{f}= ((e+\wh{f})ab -ab)\wh{f}=0$ for all $a,b\in\A$.  Since $\A$ is prime it follows that 
  $(e+\wh{f})a =a$.
  Similarly we derive   $\wh{f}b(a(e+\wh{f}) -a)=0$, and so $a(e+\wh{f}) =a$. Thus $e+\wh{f}=1_\A$.
\end{proof}

\begin{corollary}\label{C}
 A  nonzero finite dimensional prime algebra has a unity.
\end{corollary}

\begin{proof}
Again we proceed by induction on  $N=\dim_\FF \A$. If $N=1$, then $\A = \FF a$ where $0\ne a\in\A$. We have $a^2 = \lambda a$ for some  $\lambda\in \FF$. Since $\A$ is prime, $\lambda\ne 0$. But then one can check that $1_\A=\lambda^{-1}a$. Let $N > 1$. By Lemma \ref{LD}  there exists an idempotent $e\in\A$ such that $e\A e$ is a division algebra. With no loss of generality we may assume that $e$ is a nontrivial idempotent.
Define $\wh{\A}$ as in Lemma \ref{L1}. Since $\wh{\A}$ is a nonzero prime algebra and a proper subalgebra of $\A$, it has a unity by induction assumption. Consequently, $\A$ has a unity by Lemma \ref{L1}.
\end{proof}

We can therefore state a sharper version of  Theorem \ref{W1}.
 
\begin{theorem} \label{W2}
{\rm (Wedderburn's theorem for prime algebras)} Let $\A$ be a nonzero finite dimensional algebra. Then $\A$ is prime if and only if there exist a positive integer $n$ and  a  division algebra $\D$ such that
$\A\cong M_n(\D)$.
\end{theorem}

 Let us  mention that it
is more standard to state  Theorem \ref{W2}  with ``prime"  replaced by ``simple". An algebra $\A$ is said to be 
 {\em simple} if $\{0\}$ and $\A$ are its only ideals and  $ab\ne 0$ for some $a,b\in\A$. It is an easy exercise to show  that a simple algebra is prime. The converse is not true. For example,  $\FF[X]$ is prime but not simple.
However,  in the finite dimensional context the notions of primeness and simplicity coincide. This follows from 
Theorem \ref{W2} together with the easily proven fact that $M_n(\D)$ is a simple algebra. 

Let us switch to semiprime algebras. We begin with some general remarks about arbitrary algebras. 
Suppose that an algebra $\A$ has ideals $\I$ and $\J$ such that $\I + \J =\A$ and $\I\cap\J=\{0\}$. Then
 $\A\cong \I\times\J$. Indeed, $\I\cap\J =\{0\}$ yields  $xy=yx=0$ for all $x\in\I$, $y\in\J$, and this implies that $x+y\mapsto(x,y)$ is an isomorphism from $\A$ onto $\I\times\J$. Incidentally, a kind of converse also  holds: 
if $\A = \A_1\times \A_2$, then $\I = \A_1\times\{0\}$ and  $\J = \{0\}\times\A_2$ are ideals of $\A$ whose sum is equal to $\A$ and  whose intersection is trivial. 
It should be pointed out that such ideals, which one might call ``factors" in a direct product,  are considered as rather special.
Yet in the situation we will consider in the last theorem,  these  are in fact the only ideals that exist.

\begin{lemma}\label{LDP}
If an ideal $\I$ of an algebra $\A$ has a unity $1_\I$, then there exists an ideal $\J$ of $\A$ such that
$\A\cong\I\times\J$.
\end{lemma}

\begin{proof}
First note that $1_\I\A\subseteq\I$ since $1_\I\in\I$, and conversely,
$\I =1_\I\I\subseteq 1_\I\A$. Therefore $\I = 1_\I\A$. Next, since $1_\I a,a1_\I\in\I$ for every $a\in\A$, we have $1_\I a =(1_\I a)1_\I$ and $a1_\I=1_\I(a 1_\I )$.
 Comparing both relations we get $a1_\I = 1_\I a$ for all $a\in\A$. Using this we see that $\J =\{a- 1_\I a\,|\,a\in\A\}$ is also an ideal of $\A$. It is clear that $\I + \J =\A$.
 If $x\in\I\cap\J$, then $1_\I a=x=b-1_\I b$ for some $a,b\in\A$. Multiplying  by $1_\I$ we get $x=0$. Thus $\I\cap\J=\{0\}$. Now we refer to the discussion preceding the lemma.
\end{proof}

\begin{theorem} \label{W3}
{\rm (Wedderburn's theorem for semiprime algebras)} Let $\A$ be a nonzero finite dimensional algebra. Then $\A$ is semiprime if and only if there exist positive integers $n_1,\ldots,n_r$ and  division algebras $\D_1,\ldots,\D_r$ such that $\A\cong M_{n_1}(\D_1)\times\ldots\times M_{n_r}(\D_r)$.
\end{theorem}

\begin{proof}
Since  $M_{n_i}(\D_i)$ is a prime algebra, and the direct product of (semi)prime algebras is semiprime,  the ``if" part is true. The converse will be  proved by induction on $N=\dim_{\FF}\A$. The $N=1$ case is trivial, so let $N > 1$.
If $\A$ is  prime, then the result follows from Theorem \ref{W2}. 
We may therefore assume that there exists $0\ne a\in\A$ such that $\I=\{x\in\A\,|\,a\A x=\{0\}\}$ is not $\{0\}$. One easily checks that $\I$ is an ideal of $\A$. If $x\in\I$ is such that $x\I x =\{0\}$, then
  $(xax)b(xax)= x(axbxa)x=0$ for all $a,b\in\A$, which first yields $xax=0$, and hence $x=0$. Thus, $\I$ is a semiprime algebra. Since $a\notin\I$ we have $\dim_{\FF}\I < N$. By the  induction assumption we have $\I\cong M_{n_1}(\D_1)\times\ldots\times M_{n_p}(\D_p)$ for some
  $n_i\ge 1$ and  division algebras $\D_i$, $i=1,\ldots,p$.
  As each factor $M_{n_i}(\D_i)$ has a unity, so does $\I$. Lemma \ref{LDP} tells us that 
 $\A\cong\I\times\J$ for some ideal $\J$ of $\A$.
We may use the induction assumption also for $\J$ and  conclude that  $\J\cong M_{n_{p+1}}(\D_{p+1})\times\ldots\times M_{n_r}(\D_r)$ for some 
 $n_i\ge 1$ and  division algebras $\D_i$, $i=p+1,\ldots,r$. The result now clearly follows.
\end{proof}

In  the finite dimensional case  semiprime algebras coincide with the so-called {\em semisimple} algebras. In the literature one more often finds versions of 
 Theorem \ref{W3} using this term. 
\\

 {\bf Concluding remarks}. Our aim was to write an expository article presenting a shortcut from elementary definitions to a substantial piece of mathematics. The proof of Wedderburn's theorem given is certainly pretty direct.  But how original is it? To be honest, we do not know. We did not find such a proof when searching the literature. But on the other hand, it is not based on some revolutionary new idea. So many mathematicians have known this theory for so many years that one hardly imagines that something essentially new can be invented.
After a closer look at \cite{W} we have realized that a few details in our construction of matrix units are somewhat similar to those used by 
 Wedderburn himself.  Thus, some of these ideas have been  around for a hundred years. Or maybe even more. We conclude  this article by quoting Wedderburn \cite[page 78]{W}: ``Most of the results contained in the present paper have already been given, chiefly by Cartan and Frobenius, for algebras whose coefficients lie in the field of rational numbers; and it is probable that many of the methods used by these authors are capable of direct generalisation to any field. It is hoped, however, that the methods of the present paper are, in themselves and apart from the novelty of the results, sufficiently interesting to justify its publication."\\
 
 {\sc Acknowledgement.} The author would like to thank Igor Klep, Lajos Molnar, Peter \v Semrl, \v Spela \v Spenko and Ga\v sper Zadnik for reading a preliminary version of this paper and for giving several valuable suggestions and comments.

\end{document}